\documentclass[reqno]{amsart}
\usepackage{CJK}
\usepackage{hyperref}
\usepackage{cite}
\usepackage{amsmath}
\usepackage{mathrsfs}
\def\dive{\operatorname{div}}

\numberwithin{equation}{section}

\newtheorem{theorem}{Theorem}[section]
\newtheorem{lemma}[theorem]{Lemma}
\newtheorem{definition}[theorem]{Definition}

\newtheorem{proposition}[theorem]{Proposition}
\newtheorem{remark}[theorem]{Remark}

\allowdisplaybreaks

\newcommand{ \mint }{ {\int\hspace{-0.38cm}-}}

\begin{document}
	
	\title[\hfil Harnack inequality for the nonlocal equations\dots] {Harnack inequality for the nonlocal equations with general growth}
	
	\author[Y. Fang and C. Zhang  \hfil \hfilneg]{Yuzhou Fang and  Chao Zhang$^*$}
	
	\thanks{$^*$ Corresponding author.}
	
	\address{Yuzhou Fang \hfill\break School of Mathematics, Harbin Institute of Technology, Harbin 150001, China}
	\email{18b912036@hit.edu.cn}

	\address{Chao Zhang  \hfill\break School of Mathematics and Institute for Advanced Study in Mathematics, Harbin Institute of Technology, Harbin 150001, China}
	\email{czhangmath@hit.edu.cn}

	
	\maketitle
	
	\begin{abstract}
	We consider a class of generalized nonlocal $p$-Laplacian equations. We find some proper structural conditions to establish a version of nonlocal Harnack inequalities of weak solutions to such nonlocal problems by using the expansion of positivity and energy estimates.
	\end{abstract}
	
	\section{Introduction}
	\label{sec-0}
	
   Let $\Omega\subset\mathbb{R}^n \,(n\geq2)$ be a bounded domain.	In this paper, we are interested in the following class of integro-differential equations with general growth
	\begin{equation}
	\label{main}
	\mathcal{L}u=0    \quad\text{in } \Omega
	\end{equation}
	with
	$$
	\mathcal{L}u(x):=\mathrm{P.V.}\int_{\mathbb{R}^n}g\left(\frac{|u(x)-u(y)|}{|x-y|^s}\right)\frac{u(x)-u(y)}{|u(x)-u(y)|}
	\frac{K(x,y)}{|x-y|^s}\,dy,
	$$
	where the symbol $\mathrm{P.V.}$ represents ``in the principal value sense",  $s\in(0,1)$ and the function $K(x,y):\mathbb{R}^n\times\mathbb{R}^n\rightarrow(0,\infty]$ is a symmetric measurable kernel such that
	\begin{equation}
	\label{0-2}
	\frac{\Lambda^{-1}}{|x-y|^n}\leq K(x,y)\leq \frac{\Lambda}{|x-y|^n}, \quad \Lambda\geq1.
	\end{equation}
	Particularly when $\Lambda=1$, Eq. \eqref{main} is called $s$-fractional $G$-Laplace equation. The function $g:[0,\infty)\rightarrow[0,\infty)$ is continuous and strictly increasing fulfilling $g(0)=0$, $\lim_{t\rightarrow\infty}g(t)=\infty$ and
	\begin{equation}
	\label{0-3}
	1<p\leq\frac{tg(t)}{G(t)}\leq q<\infty  \quad  \text{ with }  G(t)=\int_0^tg(\tau)d\tau,
	\end{equation}
	where $G(\cdot)$ is an $N$-function possessing the $\Delta_2$ and $\nabla_2$ conditions (see Section \ref{sec-1}). 
	
	
	In recent years, a great attention has been concentrated on the nonlocal $p$-Laplacian problems, which is the special case that $g(t)=t^{p-1}$. For the regularity theory on this kind of problems, Kassmann \cite{Kas} proved the nonlocal Harnack inequality with tail-term for the fractional Laplacian. Di Castro-Kuusi-Palatucci \cite{DKP16} further investigated the local behaviour of weak solutions incorporating boundedness and H\"{o}lder continuity in the spirit of De Giorgi-Nash-Moser iteration; see also \cite{DKP14} for the nonlocal Harnack inequalities. The H\"{o}lder regularity up to the boundary was whereafter showed by Iannizzotto-Mosconi-Squassina \cite{IMS16}. We also refer the readers to \cite{BL17} for higher Sobolev regularity, \cite{KMS15} for self-improving properties, \cite{KKL19,KKP17} for the viscosity and potential theory, \cite{BP16,FP14} for fractional $p$-eigenvalue problems. When it comes to the parabolic counterpart, several features of solutions have already been studied, such as the local regularity \cite{BLS,DZZ21,Str19} and the well-posedness \cite{MRT16,Vaz16}. For more results on the nonlocal nonlinear problems of the $p$-Laplacian type, one can see for instance \cite{KKP16,KMS2015,Pal18,CCV11,FK13,Vaz21}.
	
	When $g(\cdot)$ carries a more general structure, Eq. \eqref{main} can be viewed naturally as the nonlocal analogue of the $G$-Laplace equation whose classical model is
	\begin{equation}
	\label{0-3-1}
	-\dive\left(g(|\nabla u|)\frac{\nabla u}{|\nabla u|} \right)=0 \quad\text{with } g(t)=G'(t).
	\end{equation}
	The so-called $G$-Laplace equations have been extensively studied over the past years. The regularity theory, especially for the scenario that $g(t)\approx t^p+t^q$, is initially explored by the celebrated papers of Marcellini \cite{Mar89,Mar91}. More results on the generalized $p$-Laplace equations can be found in \cite{DSV09,HHL21,Sal21}. On the other hand, Fern\'{a}ndez Bonder-Salort-Vivas \cite{FSV20} established the H\"{o}lder continuity for weak solutions to the fractional $g$-Laplacian with Dirichlet boundary values; see also \cite{FSV21} for the global regularity of eigenfunctions.  Chaker-Kim-Weidner \cite{CKW21} proved, via De Giorgi classes, the interior regularity properties for the nonlocal functionals with $(p,q)$-growth and related equations. More recently, the weak solutions to \eqref{main} were proved to be locally bounded and H\"{o}lder continuous in \cite{BKO21} under the assumption \eqref{0-3}. Regarding further studies of the nonlocal problems possessing non-standard growth, including also double phase equations and equations with variable exponents, one can refer to \cite{DeFP19,FZ21,BOS21,CK21,O21,GKS20,GKS21b,GKS21a} and references therein.
	
	Although pretty abundant research results have been obtained for the nonlocal problems with non-standard growth, to the best of our knowledge, there are few results regarding the pointwise estimates such as the Harnack inequalities. To this end, our aim of this manuscript is to investigate Harnack estimate for Eq. \eqref{main}, which can be regarded as a natural outgrowth of the result in \cite{DKP14}. Due to the possibly inhomogeneous growth of the function $G$, we have to explore the suitable conditions on $G$ in order to infer the desired result. Additionally, we need require that the Young functions $G$ satisfies the following conditions:
	\begin{equation}
	\label{0-4}
	G(t\tau)\leq c_0G(t)G(\tau)
	\end{equation}
	and
	\begin{equation}
	\label{0-5}
	\min\{t^p,t^q\}\leq G(t)\leq c_1\max\{t^p,t^q\}
	\end{equation}
	for any $t,\tau\geq0$ and $c_0,c_1$ being two positive constants. Examples of $G$ satisfying the requirements \eqref{0-3}, \eqref{0-4} and \eqref{0-5} include 
	\begin{itemize}
		\item $G(t)=t^p$, $t\ge 0$, $p>1$;
		\item $G(t)=\max\{t^p,t^q\}$, $t\ge 0$, $1<p\le q<\infty$;
		\item $G(t)=t^p+a_0t^q \text{ with } a_0>0$, $t\ge 0$, $1<p\le q<\infty$;
		\item  $G(t)=t^p\log(e+t)$, $t\ge 0$, $p>1$.
	\end{itemize}

	Before giving our main result, we introduce the so-called ``tail space",
	$$
	L^g_s(\mathbb{R}^n)=\left\{u \text{ is measurable function in } \mathbb{R}^n: \int_{\mathbb{R}^n}
	g\left(\frac{|u(x)|}{(1+|x|)^s}\right)\,\frac{dx}{(1+|x|)^{n+s}}<\infty\right\}.
	$$
	The corresponding nonlocal tail of $u$ is given by
	\begin{equation}
	\label{1}
	\mathrm{Tail}(u;x_0,R)=\int_{\mathbb{R}^n\setminus B_R(x_0)}g\left(\frac{|u(x)|}{|x-x_0|^s}\right)\,\frac{dx}{|x-x_0|^{n+s}}.
	\end{equation}
	Notice that $u\in L^g_s(\mathbb{R}^n)$ if and only if $\mathrm{Tail}(u;x_0,R)$ is finite for any $x_0\in \mathbb{R}^n$ and $R>0$. The details can be found in \cite[Subsection 2.3]{BKO21}.
	
	\medskip
	
	Now we are in a position to state the main result as follows.
	
	\begin{theorem}
		\label{thm0}
		Suppose that $s\in(0,1)$ and the assumptions \eqref{0-3}, \eqref{0-4} and \eqref{0-5} are in force. Let $u\in \mathbb{W}^{s,G}(\Omega)\cap L^g_s(\mathbb{R}^n)$ be a weak solution of Eq. \eqref{main} such that $u\geq0$ in $B_R:=B_R(x_0)\subset\Omega$. Then, for every $B_r:=B_r(x_0)\subset B_\frac{R}{2}(x_0)$, we have the following nonlocal Harnack inequality
		\begin{align*}
		\sup_{B_r}u&\leq Cr^{s(1-\frac{q}{p})\frac{q}{\epsilon}}\max_{\iota\in\{1,\frac{q}{p},\frac{p}{q}\}}\left\{\left(\inf_{B_r}u+r^sg^{-1}(r^s\mathrm{Tail}(u_-;x_0,R))\right)^\iota\right\}\\
		&\quad+Cr^sg^{-1}(r^s\mathrm{Tail}(u_-;x_0,R)),
		\end{align*}
		where  {\rm Tail$(\cdot)$} is defined in \eqref{1}, $u_-:=\max\{-u,0\}$, the positive constant $C$ depends on $n,p,q,s,\Lambda$ as well as the structural constants $c_0,c_1$ given by \eqref{0-4} and \eqref{0-5}, and the absolute constant $\epsilon\in(0,1)$, coming from Lemma \ref{lem2-3} below, is a priori determined by $n,p,q,s,\Lambda$.
	\end{theorem}
	
	\begin{remark}
	Let us point out that the extra hypotheses \eqref{0-4} and \eqref{0-5} are only exploited in the proof of Theorem \ref{thm0} below. The reason why we impose the additional strong conditions on $G$ is that we need split the term $G(u)$ into $G(u^{1-\varepsilon})G(u^\varepsilon)$ with $\varepsilon$ being an arbitrary number in $(0,1)$, and then get the integral of $u^{\varepsilon'}$ $(\varepsilon'\in(0,1))$ as the integrand, which enables us to apply Lemma \ref{lem2-3}. Observe that, if $g(t)=t^{p-1}$, then $q=p$ and
	\begin{align*}
	r^sg^{-1}(r^s\mathrm{Tail}(u_-;x_0,R))&=r^s\left(r^s\int_{\mathbb{R}^n\setminus B_R(x_0)}\frac{u_-^{p-1}(x)}{|x-x_0|^{s(p-1)}}\frac{dx}{|x-x_0|^{n+s}}\right)^\frac{1}{p-1}\\
	&=\left(\frac{r}{R}\right)^\frac{sp}{p-1}\left(R^{sp}\int_{\mathbb{R}^n\setminus B_R(x_0)}\frac{u_-^{p-1}(x)}{|x-x_0|^{n+sp}}\,dx\right)^\frac{1}{p-1}.
	\end{align*}
	Hence, our result is reduced to the Harnack inequality obtained in \cite[Theorem 1.1]{DKP14}.
   \end{remark}

	The paper is organized as follows. In Section \ref{sec-1}, we give the definition of weak solutions to Eq. \eqref{main}, and collect some notations and auxiliary inequalities to be used later. Section \ref{sec-2} is devoted to deducing infimum estimates for weak supersolutions by employing the expansion of positivity.  Finally, we prove the Harnack inequality in Section \ref{sec-3}.

	\section{Preliminaries}
	\label{sec-1}
	
	In this section, we shall give some basic inequalities, state notions of some functional spaces and weak solutions, and then provide a covering lemma.
	
	In what follows, we denote by $C$ a generic positive constant which may change from line to line. Relevant dependencies on parameters will be illustrated utilizing parentheses, i.e., $C\equiv C(n,p,q)$ means that $C$ depends on $n,p,q$. Let $B_r(x_0):=\{x\in\mathbb{R}^n:|x-x_0|<r\}$
	stand for the open ball with center $x_0$ and radius $r>0$. If not important, or clear from the context, we do not denote the center as follows: $B_r:=B_r(x_0)$. 
	If $f\in L^1(A)$ and $A\subset\mathbb{R}^n$ is a measurable subset with positive measure $0<|A|<\infty$, we denote its integral average by
	$$
	(f)_A:=\mint_Af(x)\,dx=\frac{1}{|A|}\int_Af(x)\,dx.
	$$

	The function $G:[0,\infty)\rightarrow[0,\infty)$ is an $N$-function which means that it is convex and increasing, and satisfies that
	$$
	G(0)=0, \quad \lim_{t\rightarrow0+}\frac{G(t)}{t}=0 \quad\text{and} \quad \lim_{t\rightarrow\infty}\frac{G(t)}{t}=\infty.
	$$
	The conjugate function of $N$-function $G$ is denoted by
	$$
	G^*(t)=\sup_{\tau\geq0}\{\tau t-G(\tau)\}.
	$$
   From the relation \eqref{0-3}, we now give several inequalities to be utilized later:
	\begin{itemize}
		\item[($a$)] for $t\in[0,\infty)$,
		\begin{equation}
		\label{1-1}
		\begin{cases}
		a^qG(t)\leq G(at)\leq a^pG(t) & \text{\textmd{if }} a\in(0,1),\\[2mm]
		a^pG(t)\leq G(at)\leq a^qG(t) & \text{\textmd{if }} a\in(1,\infty)
		\end{cases}
		\end{equation}
		and
		\begin{equation}
		\label{1-2}
		\begin{cases}
		a^{p'}G^*(t)\leq G^*(at)\leq a^{q'}G^*(t) & \text{\textmd{if }} a\in(0,1),\\[2mm]
		a^{q'}G^*(t)\leq G^*(at)\leq a^{p'}G^*(t) & \text{\textmd{if }} a\in(1,\infty),
		\end{cases}
		\end{equation}
		where $p',q'$ separately are the H\"{o}lder conjugates of $p,q$.
		
		\item[($b$)] Young's inequality with $\epsilon\in(0,1]$
		\begin{equation}
		\label{1-3}
		t\tau\leq \epsilon^{1-q}G(t)+\epsilon G^*(\tau), \quad t,\tau\geq0.
		\end{equation}
		
		\item[($c$)] for $t,\tau\geq0$,
		\begin{equation}
		\label{1-4}
		G^*(g(t))\leq (q-1)G(t),
		\end{equation}
		and
		\begin{equation}
		\label{1-5}
		2^{-1}(G(t)+G(\tau))\leq G(t+\tau)\leq 2^{q-1}(G(t)+G(\tau)).
		\end{equation}
		
	\end{itemize}
	
	Moreover, the function $G$ fulfills the following $\Delta_2$ and $\nabla_2$ conditions (see \cite[Proposition 2.3]{MR08}):
	\begin{itemize}
		
		\item[($\Delta_2$)] there is a constant $\mu>1$ such that $G(2t)\leq \mu G(t)$ for $t\geq0$;
		
		\smallskip
		
		\item[($\nabla_2$)] there is a constant $\nu>1$ such that $G(t)\leq \frac{1}{2\nu}G(\nu t)$ for $t\geq 0$,
		
	\end{itemize}
	where $\mu,\nu$ depend on $p,q$.
	
	We next introduce the notion of Orlicz-Sobolev spaces. For an $N$-function $G$ with the $\Delta_2$ and $\nabla_2$ conditions, the Orlicz space $L^G(\Omega)$ is defined as
	$$
	L^G(\Omega)=\left\{u \text{ is measurable function in } \Omega: \int_\Omega G(|u(x)|)\,dx<\infty\right\}
	$$
	equipped with the Luxemburg norm
	$$
	\|u\|_{L^G(\Omega)}=\inf\left\{\lambda>0: \int_\Omega G\left(\frac{|u(x)|}{\lambda}\right)\,dx\leq1\right\}.
	$$
	The fractional Orlicz-Sobolev space $W^{s,G}(\Omega)$ ($s\in(0,1)$) is given by
	$$
	W^{s,G}(\Omega)=\left\{u\in L^G(\Omega):\int_\Omega\int_\Omega G\left(\frac{|u(x)-u(y)|}{|x-y|^s}\right)\,\frac{dxdy}{|x-y|^n}<\infty\right\}
	$$
	endowed with the norm
	$$
	\|u\|_{W^{s,G}(\Omega)}=\|u\|_{L^G(\Omega)}+[u]_{s,G,\Omega},
	$$
	where $[u]_{s,G,\Omega}$ is the Gagliardo semi-norm defined as
	$$
	[u]_{s,G,\Omega}=\inf\left\{\lambda>0:\int_\Omega\int_\Omega G\left(\frac{|u(x)-u(y)|}{\lambda|x-y|^s}\right)\,\frac{dxdy}{|x-y|^n}\leq1\right\}.
	$$
	
	Let $C_{\Omega}\equiv (\Omega\times\mathbb{R}^n)\cup(\mathbb{R}^n\times\Omega)$. For  measurable function $u$ in $\mathbb R^n$, we define
	\begin{align*}
	\mathbb{W}^{s,G}(\Omega)=\left\{u\big |_{\Omega}\in L^G(\Omega) : \iint_{C_{\Omega}}G\left(\frac{|u(x)-u(y)|}{|x-y|^s}\right)\,\frac{dxdy}{|x-y|^n}<\infty\right\},
	\end{align*}
	which is the space weak solutions of \eqref{main} belong to.
	
	\medskip
	
	Now we give the definition of weak solutions to \eqref{main}.
	\begin{definition}
		We call $u\in \mathbb{W}^{s,G}(\Omega)$ a weak supersolution of Eq. \eqref{main} if
		\begin{equation}
		\label{1-6}
		\iint_{C_{\Omega}}g\left(\frac{|u(x)-u(y)|}{|x-y|^s}\right)\frac{u(x)-u(y)}{|u(x)-u(y)|}
		(\psi(x)-\psi(y))\frac{K(x,y)}{|x-y|^s}\,dxdy\geq0
		\end{equation}
		for each nonnegative function $\psi\in\mathbb{W}^{s,G}(\Omega)$ with compact support in $\Omega$. For weak subsolution, the above inequality is reversed. $u\in \mathbb{W}^{s,G}(\Omega)$ is a weak solution to \eqref{main} if and only if it is both a weak supersolution and a weak subsolution.
	\end{definition}
	
	We conclude this section with presenting the Krylov-Sofonov covering lemma (see for instance \cite{KS01}) playing an important role in proving Lemma \ref{lem2-3} below.
	
	\begin{lemma}
		\label{lem1}
		Let $\overline{\delta}\in(0,1)$ and $E\subset B_r(x_0)$ be a measurable set. Denote
		$$
		[E]_{\overline{\delta}}=\bigcup_{\rho>0}\{B_{3\rho}(x)\cap B_r(x_0), x\in B_r(x_0): |E\cap B_{3\rho}(x)|>\overline{\delta}|B_\rho(x)|\}.
		$$
		Then one of the following must hold:
		\begin{itemize}
			
			\item[(i)] $|[E]_{\overline{\delta}}|\geq \frac{c(n)}{\overline{\delta}}|E|;$
			
			\item[(ii)] $[E]_{\overline{\delta}}=B_r(x_0).$
			
		\end{itemize}
	\end{lemma}

	\section{Expansion of positivity}
	\label{sec-2}
	
	This section is devoted to deriving the infimum estimates on the weak supersolutions of \eqref{main} by expansion of positivity. The following proposition exhibits the spread of pointwise positivity in space.
	
	\begin{proposition}
		\label{pro2-1}
		Let $k\geq0$ and $u\in \mathbb{W}^{s,G}(\Omega)$ be a weak supersolution to Eq. \eqref{main} such that $u\geq0$ in $B_R(x_0)\subset\Omega$. If
		$$
		|B_r\cap\{u\geq k\}|\geq \sigma|B_r|
		$$
		for some $\sigma\in(0,1]$ and $r$ fulfilling $0<r<\frac{R}{16}\leq1$, then there is $\delta\in(0,\frac{1}{2})$, which depends on $n,p,q,s,\Lambda,\sigma$, such that
		$$
		u(x)\geq \frac{1}{2}\delta k-r^sg^{-1}(r^s\mathrm{Tail}(u_-;x_0,R)) \quad \text{in } B_{4r}.
		$$
	\end{proposition}
	
	Before proving this proposition, we first need the propagation of positivity in measure, that is the forthcoming lemma.
	
	\begin{lemma}
		\label{lem2-2}
		Let $k\geq0$ and $u\in \mathbb{W}^{s,G}(\Omega)$ be a weak supersolution to Eq. \eqref{main} such that $u\geq0$ in $B_R(x_0)\subset\Omega$. If there is a $\sigma\in(0,1]$ satisfying
		$$
		|B_r\cap\{u\geq k\}|\geq \sigma|B_r|
		$$
		with $0<r<\frac{R}{16}\leq1$, then we infer that, for any $\delta\in(0,\frac{1}{2})$,
		$$
		|B_{6r}\cap\{u\leq 2\delta k-r^sg^{-1}(r^s\mathrm{Tail}(u_-;x_0,R))\}|\leq \frac{C}{\sigma \log\frac{1}{2\delta}}|B_{6r}|
		$$
		with the constant $C>0$ depending only on $n,p,q,s,\Lambda$.
	\end{lemma}
	
	\begin{proof}
		Let $v(x):=u(x)+d$ with $d=r^sg^{-1}(r^s\mathrm{Tail}(u_-;x_0,R))$. Now take a cut-off function $\varphi\in C^\infty_0(B_{7r})$ such that
		$$
		0\leq \varphi\leq1,\quad \varphi\equiv1 \text{ in } B_{6r} \quad \text{and} \quad |\nabla\varphi|\leq\frac{c}{r}.
		$$
		We select $\eta:=\varphi^q\frac{v}{G(v/r^s)}$ as the test function in the weak formulation \eqref{1-6}, and then slightly modify the expression to have
		\begin{align}
		\label{2-1}
		0&\leq \int_{B_{8r}}\int_{B_{8r}}g\left(\frac{|v(x)-v(y)|}{|x-y|^s}\right)\frac{v(x)-v(y)}{|v(x)-v(y)|}\left(\frac{v(x)\varphi^q(x)}{G(v(x)/r^s)}
		-\frac{v(y)\varphi^q(y)}{G(v(y)/r^s)}\right)\frac{K(x,y)}{|x-y|^s}\,dxdy \nonumber\\
		&\quad+2\int_{\mathbb{R}^n\setminus B_{8r}}\int_{B_{8r}}g\left(\frac{|v(x)-v(y)|}{|x-y|^s}\right)\frac{v(x)-v(y)}{|v(x)-v(y)|}\frac{v(x)\varphi^q(x)}{G(v(x)/r^s)}
		\frac{K(x,y)}{|x-y|^s}\,dxdy  \nonumber\\
		&=:I_1+2I_2.
		\end{align}
		Following the arguments of Steps 1--3 in \cite[Proposition 3.4]{BKO21}, we get
		\begin{equation}
		\label{2-2}
		I_1\leq-\frac{1}{C}\int_{B_{6r}}\int_{B_{6r}}\left|\log\frac{v(x)}{v(y)}\right|\,\frac{dxdy}{|x-y|^n}+Cr^n.
		\end{equation}
	
		For the integral $I_2$,
		\begin{align*}
		I_2&=\int_{\mathbb{R}^n\setminus B_{8r}\cap\{v(y)<0\}}\int_{B_{8r}}g\left(\frac{|v(x)-v(y)|}{|x-y|^s}\right)\frac{v(x)-v(y)}{|v(x)-v(y)|}\frac{v(x)\varphi^q(x)}{G(v(x)/r^s)}
		\frac{K(x,y)}{|x-y|^s}\,dxdy\\
		&\quad+\int_{\mathbb{R}^n\setminus B_{8r}\cap\{v(y)\geq0\}}\int_{B_{8r}}g\left(\frac{|v(x)-v(y)|}{|x-y|^s}\right)\frac{v(x)-v(y)}{|v(x)-v(y)|}\frac{v(x)\varphi^q(x)}{G(v(x)/r^s)}
		\frac{K(x,y)}{|x-y|^s}\,dxdy\\
		&=:I_{21}+I_{22}.
		\end{align*}
		We first evaluate $I_{21}$. Note that, by \eqref{0-3}, \eqref{1-1} and \eqref{1-5},
		\begin{align*}
		g\left(\frac{v(x)+v_-(y)}{|x-y|^s}\right)&\leq q2^{q-1}\frac{G\left(\frac{v(x)}{|x-y|^s}\right)+G\left(\frac{v_-(y)}{|x-y|^s}\right)}{\frac{v(x)+v_-(y)}{|x-y|^s}}\\
		&\leq
		\frac{q2^{q-1}}{p}\left(g\left(\frac{v(x)}{|x-y|^s}\right)+g\left(\frac{v_-(y)}{|x-y|^s}\right)\right).
		\end{align*}
		Then,
		\begin{align*}
		I_{21}&=\int_{\mathbb{R}^n\setminus B_{8r}\cap\{v(y)<0\}}\int_{B_{8r}}g\left(\frac{|v(x)-v(y)|}{|x-y|^s}\right)\frac{v(x)\varphi^q(x)}{G(v(x)/r^s)}
		\frac{K(x,y)}{|x-y|^s}\,dxdy\\
		&\leq C\int_{\mathbb{R}^n\setminus B_{8r}}\int_{B_{8r}}\left(g\left(\frac{v(x)}{|x-y|^s}\right)+g\left(\frac{v_-(y)}{|x-y|^s}\right)\right)\frac{\varphi^q(x)r^s}{g(v(x)/r^s)}
		\frac{K(x,y)}{|x-y|^s}\,dxdy\\
		&\leq C\int_{\mathbb{R}^n\setminus B_{8r}}\int_{B_{7r}}\left(g\left(\frac{v(x)}{r^s}\right)+g\left(\frac{v_-(y)}{|x-y|^s}\right)\right)\frac{r^s}{g(v(x)/r^s)|x-y|^{n+s}}
		\,dxdy.
		\end{align*}
		When $x\in B_{7r}$ and $y\in \mathbb{R}^n\setminus B_{8r}$,
		$$
		|y-x_0|\leq \left(1+\frac{|x-x_0|}{|y-x|}\right)|y-x|\leq 8|y-x|,
		$$
		we further get
		\begin{align*}
		I_{21}&\leq Cr^s\int_{\mathbb{R}^n\setminus B_{8r}}\int_{B_{7r}}\,\frac{dxdy}{|y-x_0|^{n+s}}+C\frac{r^s}{g(d/r^s)}\int_{\mathbb{R}^n\setminus B_{8r}}\int_{B_{7r}}g\left(\frac{u_-(y)}{|x_0-y|^s}\right)\,\frac{dxdy}{|y-x_0|^{n+s}}\\
		&\leq Cr^n+C\frac{r^{n+s}}{g(d/r^s)}\mathrm{Tail}(u_-;x_0,R),
		\end{align*}
		where we utilized the fact $u(y)\geq0$ in $B_R\supset B_{8r}$. Here the constant $C$ depends on $n,p,q,s,\Lambda$. 
		We next estimate $I_{22}$ as
		\begin{align*}
		I_{22}&\leq\Lambda\int_{\mathbb{R}^n\setminus B_{8r}\cap\{v(y)\geq0\}}\int_{B_{8r}\cap\{v(x)>v(y)\}}g\left(\frac{|v(x)-v(y)|}{|x-y|^s}\right)\frac{v(x)\varphi^q(x)}{G(v(x)/r^s)}
		\frac{1}{|x-y|^{n+s}}\,dxdy\\
		&\leq \Lambda\int_{\mathbb{R}^n\setminus B_{8r}}\int_{B_{8r}}g\left(\frac{v(x)}{|x-y|^s}\right)\frac{r^s\varphi^q(x)v(x)/r^s}{G(v(x)/r^s)}
		\frac{1}{|x-y|^{n+s}}\,dxdy\\
		&\leq C\int_{\mathbb{R}^n\setminus B_{8r}}\int_{B_{7r}}\frac{r^s}{|x-y|^{n+s}}\,dxdy\\
		&\leq Cr^n.
		\end{align*}
		Recalling the definition of $d$, we arrive at
		\begin{equation}
		\label{2-3}
		I_2\leq Cr^n
		\end{equation}
		with $C$ depending on $n,p,q,s,\Lambda$. 
		
		Merging \eqref{2-2}, \eqref{2-3} with \eqref{2-1} yields that
		$$
		\int_{B_{6r}}\int_{B_{6r}}\left|\log\frac{v(x)}{v(y)}\right|\,\frac{dxdy}{|x-y|^n}\leq Cr^n.
		$$
		For all $\delta\in(0,\frac{1}{2})$, set
		$$
		w:=\left[\min\left\{\log\frac{1}{2\delta},\log\frac{k+d}{v}\right\}\right]_+.
		$$
		Owing to $w$ being a truncation of $\log(k+d)-\log v$, there holds that
		$$
		\int_{B_{6r}}\int_{B_{6r}}\frac{|w(x)-w(y)|}{|x-y|^n}\,dxdy\leq\int_{B_{6r}}\int_{B_{6r}}\left|\log\frac{v(x)}{v(y)}\right|\,\frac{dxdy}{|x-y|^n}
		\leq Cr^n.
		$$
		Observe that
		$$
		\int_{B_{6r}}|w(x)-(w)_{B_{6r}}|\,dx\leq C(n)\int_{B_{6r}}\int_{B_{6r}}\frac{|w(x)-w(y)|}{|x-y|^n}\,dxdy.
		$$
		Hence,
		$$
		\int_{B_{6r}}|w(x)-(w)_{B_{6r}}|\,dx\leq Cr^n=C|B_{6r}|.
		$$
		In the same way as the computations in \cite[page 1819]{DKP14}, we finally deduce that
		$$
		|\{u\leq2\delta k-d\}\cap B_{6r}|\leq |\{v\leq2\delta k+2\delta d\}\cap B_{6r}|\leq \frac{C}{\sigma \log\frac{1}{2\delta}}|B_{6r}|.
		$$
		We now have finished the proof.
	\end{proof}
	
	Based on the above lemma, we can conclude the proof of Proposition \ref{pro2-1}.
	
	\medskip
	
	\noindent\textbf{Proof of Proposition \ref{pro2-1}.} We may suppose, with no loss of generality, that
	$$
	\frac{1}{2}\delta k>r^sg^{-1}(r^s\mathrm{Tail}(u_-;x_0,R)).
	$$
	We now choose a cut-off function $\varphi\in C^\infty_0(B_\rho)$ with $4r\leq \rho\leq6r$ and take the test function $\eta=v_-\varphi^q:=(l-u)_+\varphi^q$ for $l\in(\frac{1}{2}\delta k,2\delta k)$ in the weak formulation \eqref{1-6}. Then we have
	\begin{align}
	\label{2-4}
	0&\leq \int_{B_\rho}\int_{B_\rho}g\left(\frac{|u(x)-u(y)|}{|x-y|^s}\right)\frac{u(x)-u(y)}{|u(x)-u(y)|}(v_-(x)\varphi^q(x)
	-v_-(y)\varphi^q(y))\frac{K(x,y)}{|x-y|^s}\,dxdy \nonumber\\
	&\quad+2\int_{\mathbb{R}^n\setminus B_\rho}\int_{B_\rho}g\left(\frac{|u(x)-u(y)|}{|x-y|^s}\right)\frac{u(x)-u(y)}{|u(x)-u(y)|}v_-(x)\varphi^q(x)
	\frac{K(x,y)}{|x-y|^s}\,dxdy  \nonumber\\
	&=:I_1+2I_2.
	\end{align}
	
	We first evaluate $I_2$,
	\begin{align*}
	I_2&\leq \int_{\mathbb{R}^n\setminus B_\rho\cap\{u(y)<0\}}\int_{B_\rho}g\left(\frac{u(x)-u(y)}{|x-y|^s}\right)(l-u(x))_+\varphi^q(x)
	\frac{K(x,y)}{|x-y|^s}\,dxdy\\
	&\quad +\int_{\mathbb{R}^n\setminus B_\rho\cap\{u(y)\geq0\}}\int_{B_\rho}g\left(\frac{|u(x)-u(y)|}{|x-y|^s}\right)\frac{u(x)-u(y)}{|u(x)-u(y)|}(l-u(x))_+\varphi^q(x)
	\frac{K(x,y)}{|x-y|^s}\,dxdy\\
	&\leq 2l\int_{\mathbb{R}^n\setminus B_\rho}\int_{B_\rho}g\left(\frac{l+u_-(y)}{|x-y|^s}\right)\chi_{\{u<l\}}(x)\varphi^q(x)
	\frac{K(x,y)}{|x-y|^s}\,dxdy\\
	&\leq Cl|B_\rho\cap\{u<l\}|\sup_{x\in\mathrm{supp}\,\varphi}\int_{\mathbb{R}^n\setminus B_\rho}g\left(\frac{l+u_-(y)}{|x-y|^s}\right)
	\frac{K(x,y)}{|x-y|^s}\,dy.
	\end{align*}
	We proceed with treating the integral $I_1$. This procedure is similar to the estimate on $I$ in \cite[Proposition 3.1]{BKO21}, but for the sake of readability we give a sketched proof. Assume $u(x)\geq u(y)$. Then
	\begin{align*}
	&\quad g\left(\frac{|u(x)-u(y)|}{|x-y|^s}\right)\frac{u(x)-u(y)}{|u(x)-u(y)|}(v_-(x)\varphi^q(x)
	-v_-(y)\varphi^q(y))\\
	&\leq-g\left(\frac{|v_-(x)-v_-(y)|}{|x-y|^s}\right)\frac{v_-(x)-v_-(y)}{|v_-(x)-v_-(y)|}(v_-(x)\varphi^q(x)
	-v_-(y)\varphi^q(y)),
	\end{align*}
	by distinguishing three cases that $l\geq u(x)\geq u(y)$, $u(x)\geq l>u(y)$ and $u(x)\geq u(y)\geq l$. Exchanging the roles of $x$ and $y$, we in general case also have the previous inequality. We next consider two cases:
	\begin{equation*}
	\begin{cases}
	\textbf{Case 1: }v_-(x)>v_-(y) & \text{\textmd{and}} \quad \varphi(x)\leq\varphi(y),\\[2mm]
	\textbf{Case 2: }v_-(x)>v_-(y) & \text{\textmd{and}} \quad \varphi(x)>\varphi(y).
	\end{cases}
	\end{equation*}
	The case that $v_-(x)\leq v_-(y)$ is symmetric. In Case 1, by \eqref{0-3} and Young's inequality \eqref{1-3},
	\begin{align*}
	&\quad-g\left(\frac{|v_-(x)-v_-(y)|}{|x-y|^s}\right)\frac{v_-(x)-v_-(y)}{|v_-(x)-v_-(y)|}\frac{(v_-(x)\varphi^q(x)
		-v_-(y)\varphi^q(y))}{|x-y|^s}\\
	&=-g\left(\frac{v_-(x)-v_-(y)}{|x-y|^s}\right)\frac{v_-(x)-v_-(y)}{|x-y|^s}\varphi^q(y)+
	g\left(\frac{v_-(x)-v_-(y)}{|x-y|^s}\right)\frac{\varphi^q(y)-\varphi^q(x)}{|x-y|^s}v_-(x)\\
	&\leq-pG\left(\frac{v_-(x)-v_-(y)}{|x-y|^s}\right)\varphi^q(y)+q
	g\left(\frac{v_-(x)-v_-(y)}{|x-y|^s}\right)\varphi^{q-1}(y)\frac{\varphi(y)-\varphi(x)}{|x-y|^s}v_-(x)\\
	&\leq -pG\left(\frac{v_-(x)-v_-(y)}{|x-y|^s}\right)\varphi^q(y)+\epsilon q(q-1)G\left(\frac{v_-(x)-v_-(y)}{|x-y|^s}\right)\varphi^q(y)\\
	&\quad+C(\epsilon)qG\left(\frac{\varphi(y)-\varphi(x)}{|x-y|^s}v_-(x)\right)\\
	&\leq-\frac{p}{2}G\left(\frac{v_-(x)-v_-(y)}{|x-y|^s}\right)\varphi^q(y)+C(p,q)G\left(\frac{\varphi(y)-\varphi(x)}{|x-y|^s}v_-(x)\right),
	\end{align*}
	where we take $\epsilon=\frac{p}{2q(q-1)}$. In the other case,
	\begin{align*}
	&\quad-g\left(\frac{|v_-(x)-v_-(y)|}{|x-y|^s}\right)\frac{v_-(x)-v_-(y)}{|v_-(x)-v_-(y)|}\frac{(v_-(x)\varphi^q(x)
		-v_-(y)\varphi^q(y))}{|x-y|^s}\\
	&\leq-g\left(\frac{v_-(x)-v_-(y)}{|x-y|^s}\right)\frac{v_-(x)-v_-(y)}{|x-y|^s}\varphi^q(x)+
	g\left(\frac{v_-(x)-v_-(y)}{|x-y|^s}\right)\frac{\varphi^q(y)-\varphi^q(x)}{|x-y|^s}v_-(y)\\
	&\leq-pG\left(\frac{v_-(x)-v_-(y)}{|x-y|^s}\right)\varphi^q(x).
	\end{align*}
	
	All in all, we derive
	\begin{align*}
	&\quad-g\left(\frac{|v_-(x)-v_-(y)|}{|x-y|^s}\right)\frac{v_-(x)-v_-(y)}{|v_-(x)-v_-(y)|}\frac{(v_-(x)\varphi^q(x)
		-v_-(y)\varphi^q(y))}{|x-y|^s}\\
	&\leq-C(p)G\left(\frac{|v_-(x)-v_-(y)|}{|x-y|^s}\right)\min\{\varphi^q(x),\varphi^q(y)\}\\
	&\quad+C(p,q)G\left(\frac{|\varphi(y)-\varphi(x)|}{|x-y|^s}\max\{v_-(x),v_-(y)\}\right).
	\end{align*}
	Therefore,
	\begin{align*}
	I_1&\leq-C\int_{B_\rho}\int_{B_\rho}G\left(\frac{|v_-(x)-v_-(y)|}{|x-y|^s}\right)\min\{\varphi^q(x),\varphi^q(y)\}K(x,y)\,dxdy\\
	&\quad+C\int_{B_\rho}\int_{B_\rho}G\left(\frac{|\varphi(x)-\varphi(y)|}{|x-y|^s}\max\{v_-(x),v_-(y)\}\right)K(x,y)\,dxdy.
	\end{align*}
	Combining the estimates on $I_1$ and $I_2$ with \eqref{2-4}, we know that
	\begin{align}
	\label{2-5}
	&\quad \int_{B_\rho}\int_{B_\rho}G\left(\frac{|v_-(x)-v_-(y)|}{|x-y|^s}\right)\min\{\varphi^q(x),\varphi^q(y)\}K(x,y)\,dxdy \nonumber\\
	&\leq C\int_{B_\rho}\int_{B_\rho}G\left(\frac{|\varphi(x)-\varphi(y)|}{|x-y|^s}\max\{v_-(x),v_-(y)\}\right)K(x,y)\,dxdy \nonumber\\
	&\quad+Cl|B_\rho\cap\{u<l\}|\sup_{x\in\mathrm{supp}\,\varphi}\int_{\mathbb{R}^n\setminus B_\rho}g\left(\frac{l+u_-(y)}{|x-y|^s}\right)
	\frac{1}{|x-y|^{n+s}}\,dy.
	\end{align}
	
	Next, we will perform an iteration process. Set
	\begin{align*}
	&l_j=\left(\frac{1}{2}+\frac{1}{2^{j+1}}\right)\delta k, \quad \rho_j=4r+\frac{1}{2^{j-1}}r,\\
	&B_j:=B_{\rho_j}(x_0), \quad \tilde{\rho}_j=\frac{\rho_j+\rho_{j+1}}{2}, \quad v_j=(l_j-u)_+
	\end{align*}
	for $j=0,1,2,\cdots$. We can find that
	$$
	4r\leq\rho_j,\tilde{\rho}_j\leq6r, \quad l_j-l_{j+1}=\frac{1}{2^{j+2}}\delta k\geq \frac{1}{2^{j+2}}l_j
	$$
	and
	\begin{equation}
	\label{2-6}
	v_j\geq(l_j-l_{j+1})\chi_{\{u<l_{j+1}\}}\geq2^{-j-2}l_j\chi_{\{u<l_{j+1}\}}.
	\end{equation}
	Take cut-off functions $\varphi_j\in C^\infty_0(B_{\tilde{\rho}_j}(x_0))$ ($j=0,1,2,\cdots$) such that
	$$
	0\leq \varphi_j\leq1, \quad \varphi_j\equiv1 \text{ in } B_{j+1} \quad\text{and}\quad |\nabla\varphi_j|\leq c\frac{2^j}{r}.
	$$
	With $v_-,\varphi,l,\rho$ replaced by $v_j,\varphi_j,l_j,\rho_j$ respectively, \eqref{2-5} turns into
	\begin{align*}
	&\quad \int_{B_j}\int_{B_j}G\left(\frac{|v_j(x)-v_j(y)|}{|x-y|^s}\right)\min\{\varphi_j^q(x),\varphi_j^q(y)\}K(x,y)\,dxdy\\
	&\leq C\int_{B_j}\int_{B_j}G\left(\frac{|\varphi_j(x)-\varphi_j(y)|}{|x-y|^s}\max\{v_j(x),v_j(y)\}\right)K(x,y)\,dxdy \nonumber\\
	&\quad+Cl_j|B_j\cap\{u<l_j\}|\sup_{x\in\mathrm{supp}\,\varphi_j}\int_{\mathbb{R}^n\setminus B_j}g\left(\frac{l_j+u_-(y)}{|x-y|^s}\right)
	\frac{1}{|x-y|^{n+s}}\,dy\\
	&\leq C2^{qj}\int_{B_j}\int_{B_j}\left(\frac{|x-y|}{r}\right)^{(1-s)p}G\left(\frac{\max\{v_j(x),v_j(y)\}}{r^s}\right)\,\frac{dxdy}{|x-y|^n}\\
	&\quad+C2^{j(n+qs)}l_j|B_j\cap\{u<l_j\}|\int_{\mathbb{R}^n\setminus B_j}g\left(\frac{l_j+u_-(y)}{|x_0-y|^s}\right)
	\frac{1}{|x_0-y|^{n+s}}\,dy\\
	&=:J_1+J_2,
	\end{align*}
	where we have employed \eqref{1-1} and the facts that
	$$
	|\varphi_j(x)-\varphi_j(y)|\leq C\frac{2^j}{r}|x-y|
	$$
	and for $x\in\mathrm{supp}\,\varphi_j\subset B_{\tilde{\rho}_j}$ and $y\in\mathbb{R}^n\setminus B_j$,
	$$
	|y-x_0|\leq\left(1+\frac{\tilde{\rho}_j}{\rho_j-\tilde{\rho}_j}\right)|y-x|\leq 2^{j+4}|y-x|.
	$$
	Observe that, by \eqref{0-3} and \eqref{1-5},
	$$
	g\left(\frac{l_j+u_-(y)}{|x_0-y|^s}\right)\leq q\frac{G\left(\frac{l_j+u_-(y)}{|x_0-y|^s}\right)}{\frac{l_j+u_-(y)}{|x_0-y|^s}}\leq \frac{q2^{q-1}}{p}\left(g\left(\frac{l_j}{|x_0-y|^s}\right)+g\left(\frac{u_-(y)}{|x_0-y|^s}\right)\right).
	$$
	Since $u(y)\geq0$ in $B_R$, for the integral in $J_2$ there holds that
	\begin{align*}
	&\quad\int_{\mathbb{R}^n\setminus B_j}g\left(\frac{l_j+u_-(y)}{|x_0-y|^s}\right)
	\frac{1}{|x_0-y|^{n+s}}\,dy\\
	&\leq C\int_{\mathbb{R}^n\setminus B_j}g\left(\frac{l_j}{\rho_j^s}\right)+g\left(\frac{u_-(y)}{|x_0-y|^s}\right)\,\frac{dy}{|x_0-y|^{n+s}}\\
	&\leq Cr^{-s}g(l_j/r^s)+C\mathrm{Tail}(u_-;x_0,R)\\
	&\leq Cr^{-s}g(l_j/r^s),
	\end{align*}
	where in the last inequality we note that
	$$
	l_j>\frac{1}{2}\delta k\geq r^sg^{-1}(r^s\mathrm{Tail}(u_-;x_0,R)),
	$$
	namely,
	$$
	\mathrm{Tail}(u_-;x_0,R)\leq r^{-s}g(l_j/r^s).
	$$
	As for $J_1$, it follows from \eqref{1-5} that
	\begin{align*}
	J_1&\leq C2^{qj}\int_{B_j}\int_{B_j}\left(\frac{|x-y|}{r}\right)^{(1-s)p}G\left(\frac{v_j(x)}{r^s}\right)\,\frac{dxdy}{|x-y|^n}\\
	&\leq C2^{qj}r^{-(1-s)p}\int_{B_j}G(v_j(x)/r^s)\,dx\int_{B_{2\rho_j}(x)}\,\frac{dy}{|x-y|^{n-(1-s)p}}\\
	&\leq C2^{qj}\int_{B_j}G(v_j(x)/r^s)\,dx\\
	&\leq C2^{qj}G(l_j/r^s)|B_j\cap\{u<l_j\}|.
	\end{align*}
	Putting together these preceding estimates yields that
	\begin{align*}
	&\quad \int_{B_{j+1}}\int_{B_{j+1}}G\left(\frac{|v_j(x)-v_j(y)|}{|x-y|^s}\right)K(x,y)\,dxdy\\
	&\leq C2^{j(n+sq+q)}G(l_j/r^s)|B_j\cap\{u<l_j\}|.
	\end{align*}
	According to Lemma 4.1 in \cite{BKO21}, we obtain
	\begin{align}
	\label{2-7}
	&\quad \left(\mint_{B_{j+1}}G^\theta\left(\frac{|v_j-(v_j)_{B_{j+1}}|}{\rho_{j+1}^s}\right)\,dx\right)^\frac{1}{\theta} \nonumber\\
	&\leq C\mint_{B_{j+1}}\int_{B_{j+1}}G\left(\frac{|v_j(x)-v_j(y)|}{|x-y|^s}\right)\,\frac{dxdy}{|x-y|^n} \nonumber\\
	&\leq C2^{j(n+sq+q)}G(l_j/r^s)\frac{|B_j\cap\{u<l_j\}|}{|B_j|},
	\end{align}
	with $\theta>1$ depending only on $n,s$.
	
	On the other hand, by means of \eqref{1-5} and Jensen's inequality, the following display
	\begin{equation}
	\label{2-8}
	\begin{split}
	\left(\mint_{B_{j+1}}G^\theta\left(\frac{v_j}{\rho_{j+1}^s}\right)\,dx\right)^\frac{1}{\theta}&\leq C\left(\mint_{B_{j+1}}G^\theta\left(\frac{|v_j-(v_j)_{B_{j+1}}|}{\rho_{j+1}^s}\right)\,dx\right)^\frac{1}{\theta}\\
	&\quad+C\mint_{B_{j+1}}G\left(\frac{v_j}{\rho_{j+1}^s}\right)\,dx
	\end{split}
	\end{equation}
	is valid. Moreover, via \eqref{2-6} and $4r\leq \rho_{j+1}\leq6r$,
	\begin{equation}
	\label{2-9}
	G^\theta(v_j/\rho_{j+1}^s)\geq G^\theta(2^{-j-2}l_j/\rho_{j+1}^s)\chi_{\{u<l_{j+1}\}}\geq C2^{-jq\theta}G(l_j/r^s)\chi_{\{u<l_{j+1}\}}.
	\end{equation}
	It then follows from \eqref{2-7}--\eqref{2-9} that
	\begin{align*}
	&\quad C2^{-jq}G(l_j/r^s)\left(\mint_{B_{j+1}}\chi_{\{u<l_{j+1}\}}\,dx\right)^\frac{1}{\theta}\\
	&\leq C2^{j(n+sq+q)}G(l_j/r^s)\frac{|B_j\cap\{u<l_j\}|}{|B_j|}+C\mint_{B_{j+1}}G(l_j/r^s)\chi_{\{u<l_{j}\}}\,dx\\
	&\leq C2^{j(n+sq+q)}G(l_j/r^s)\frac{|B_j\cap\{u<l_j\}|}{|B_j|},
	\end{align*}
	that is,
	$$
	\left(\frac{|B_{j+1}\cap\{u<l_{j+1}\}|}{|B_{j+1}|}\right)^\frac{1}{\theta}\leq C2^{j(n+sq+2q)}\frac{|B_j\cap\{u<l_j\}|}{|B_j|}.
	$$
	
	Denote
	$$
	A_j=\frac{|B_j\cap\{u<l_j\}|}{|B_j|}.
	$$
	Then
	$$
	A_{j+1}\leq C2^{j(n+sq+2q)\theta}A_j^\theta.
	$$
	We can apply the iteration lemma (see, e.g., \cite[Lemma 7.1]{Giu}) to deduce that if
	$$
	A_0\leq C^\frac{-1}{\theta-1}2^{-(n+sq+2q)\frac{\theta}{(\theta-1)^2}}=:\beta,
	$$
	then $A_j\rightarrow0$ as $j\rightarrow\infty$. Now from Lemma \ref{lem2-2} we examine
	\begin{align*}
	A_0&=\frac{|B_{6r}\cap\{u<\delta k\}|}{|B_{6r}|}\\
	&\leq \frac{|B_{6r}\cap\{u\leq 2\delta k-r^sg^{-1}(r^s\mathrm{Tail}(u_-;x_0,R))\}|}{|B_{6r}|}\\
	&\leq \frac{C}{\sigma \log\frac{1}{2\delta}}.
	\end{align*}
	As long as we choose such small $\delta$ that
	$$
	\frac{C}{\sigma \log\frac{1}{2\delta}}\leq \beta \quad\Rightarrow\quad \delta\leq \frac{1}{2}e^{-\frac{C}{\sigma\beta}}<\frac{1}{2},
	$$
	the desired result $\lim_{j\rightarrow\infty}A_j=0$ can be justified. In other words, we draw a conclusion that there exists $\delta$, determined by $n,p,q,s,\Lambda$ and $\sigma$, such that
	$$
	u(x)\geq \frac{1}{2}\delta k
	$$
	in $B_{4r}$. We now complete the proof.
	
	\medskip
	
	At the end of this section, as a consequence of Proposition \ref{pro2-1} and the Krylov-Sofonov covering lemma, we derive the following result.
	
	\begin{lemma}
		\label{lem2-3}
		Suppose that $u\in \mathbb{W}^{s,G}(\Omega)$, satisfying $u\geq0$ in $B_R(x_0)\subset\Omega$, is a weak supersolution to Eq. \eqref{main}. Then we can find two constants $\epsilon\in(0,1)$ and $C\geq1$, both of which depend only upon $n,p,q,s,\Lambda$, such that, when $B_r(x_0)\subset B_R(x_0)$,
		$$
		\left(\mint_{B_r}u^\epsilon\,dx\right)^\frac{1}{\epsilon}\leq C\inf_{B_r}u+Cr^sg^{-1}(r^s\mathrm{Tail}(u_-;x_0,R))
		$$
		is valid.
	\end{lemma}
	
	\begin{proof}
		Define for any $t>0$
		$$
		A^i_t=\left\{x\in B_r: u(x)>t\left(\frac{1}{2}\delta\right)^i-\frac{T}{1-\delta/2}\right\},  \quad i=0,1,2,\cdots,
		$$
		where $\delta$ is identical to that of Proposition \ref{pro2-1}, and $T$ stands for
		$$
		T=r^sg^{-1}(r^s\mathrm{Tail}(u_-;x_0,R)).
		$$
		Recalling Lemma \ref{lem1} and Proposition \ref{pro2-1}, we could follow the proof of \cite[Lemma 4.1]{DKP14} verbatim, except substituting $\delta$ in \cite[Lemma 4.1]{DKP14} with $\frac{1}{2}\delta$ here, to arrive at
		$$
		\mint_{B_r}u^\epsilon\,dx\leq C\left(\inf_{B_r}u+\frac{T}{1-\delta/2}\right)^\epsilon.
		$$
		This directly implies the desired result.
	\end{proof}
	
	\section{Nonlocal Harnack inequality}
	\label{sec-3}
	
	In this section, we are going to show the nonlocal Harnack inequality by merging the local boundedness on subsolutions (Lemma \ref{lem3-2}) along with the infimum estimate of supersolution (Lemma \ref{lem2-3}), and taking into account the tail estimate for solutions (Lemma \ref{lem3-1}) in a suitable way.
	
	\begin{lemma}
		\label{lem3-1}
		Assume that $u\in \mathbb{W}^{s,G}(\Omega)\cap L^g_s(\mathbb{R}^n)$ is a weak solution to Eq. \eqref{main} such that $u\geq0$ in $B_R(x_0)\subset\Omega$. Then the tail estimate
		$$
		r^sg^{-1}(r^s\mathrm{Tail}(u_+;x_0,r))\leq C\sup_{B_r}u+Cr^sg^{-1}(r^s\mathrm{Tail}(u_-;x_0,R))
		$$
		holds true for all $0<r<R$, where $C>0$ depends only on $n,p,q,s,\Lambda$.
	\end{lemma}
	
	\begin{proof}
		Let $l=\sup_{B_r}u$. We take the test function
		$$
		\eta:=(u-2l)\varphi^q
		$$
		in the weak formulation \eqref{1-6}, where $\varphi\in C^\infty_0(B_r)$ satisfies that
		$$
		0\leq\varphi\leq1, \quad \varphi\equiv1 \text{ in } B_\frac{r}{2}, \quad \varphi\equiv0 \text{ on } \mathbb{R}^n\setminus B_\frac{3r}{4} \quad \text{and} \quad |\nabla \varphi|\leq \frac{c}{r},
		$$
		to derive
		\begin{align}
		\label{3-1}
		0&= \int_{B_r}\int_{B_r}g\left(\frac{|u(x)-u(y)|}{|x-y|^s}\right)\frac{u(x)-u(y)}{|u(x)-u(y)|}(\eta(x)
		-\eta(y))\frac{K(x,y)}{|x-y|^s}\,dxdy \nonumber\\
		&\quad+2\int_{\mathbb{R}^n\setminus B_r}\int_{B_r}g\left(\frac{|u(x)-u(y)|}{|x-y|^s}\right)\frac{u(x)-u(y)}{|u(x)-u(y)|}\eta(x)
		\frac{K(x,y)}{|x-y|^s}\,dxdy  \nonumber\\
		&=:I_1+2I_2.
		\end{align}
		For $I_2$, we can see that
		\begin{align*}
		I_2&\geq \int_{\mathbb{R}^n\setminus B_r\cap\{u(y)\geq l\}}\int_{B_r}g\left(\frac{u(y)-u(x)}{|x-y|^s}\right)(2l-u(x))\varphi^q(x)\frac{K(x,y)}{|x-y|^s}\,dxdy\\
		&\quad-\int_{\mathbb{R}^n\setminus B_r\cap\{u(y)< l\}}\int_{B_r}2lg\left(\frac{|u(y)-u(x)|}{|x-y|^s}\right)\varphi^q(x)\frac{K(x,y)}{|x-y|^s}\,dxdy\\
		&\geq \int_{\mathbb{R}^n\setminus B_r}\int_{B_r}lg\left(\frac{(u(y)-l)_+}{|x-y|^s}\right)\varphi^q(x)\frac{K(x,y)}{|x-y|^s}\,dxdy\\
		&\quad-\int_{\mathbb{R}^n\setminus B_r}\int_{B_r}2lg\left(\frac{|u(y)-u(x)|}{|x-y|^s}\right)\chi_{\{u(y)<l\}}\varphi^q(x)\frac{K(x,y)}{|x-y|^s}\,dxdy\\
		&=:I_{21}-I_{22}.
		\end{align*}
		We know from \eqref{0-3} and \eqref{1-5} that
		$$
		g\left(\frac{u_+(y)}{|x-y|^s}\right)\leq g\left(\frac{(u(y)-l)_++l}{|x-y|^s}\right)\leq C\left(g\left(\frac{(u(y)-l)_+}{|x-y|^s}\right)+g\left(\frac{l}{|x-y|^s}\right)\right).
		$$
		Thereby,
		\begin{align*}
		I_{21}
		&\geq Cl\int_{\mathbb{R}^n\setminus B_r}\int_{B_r}g\left(\frac{u_+(y)}{|x-y|^s}\right)\varphi^q(x)\frac{K(x,y)}{|x-y|^s}\,dxdy\\
		&\quad-l\int_{\mathbb{R}^n\setminus B_r}\int_{B_r}g\left(\frac{l}{|x-y|^s}\right)\varphi^q(x)\frac{K(x,y)}{|x-y|^s}\,dxdy\\
		&\geq Cl\int_{\mathbb{R}^n\setminus B_r}\int_{B_\frac{r}{2}}g\left(\frac{u_+(y)}{|x-y|^s}\right)\frac{1}{|x-y|^{n+s}}\,dxdy\\
		&\quad-Cl\int_{\mathbb{R}^n\setminus B_r}\int_{B_\frac{3r}{4}}g\left(\frac{l}{|x-y|^s}\right)\frac{1}{|x-y|^{n+s}}\,dxdy\\
		&\geq Cl\int_{\mathbb{R}^n\setminus B_r}\int_{B_\frac{r}{2}}g\left(\frac{u_+(y)}{|x_0-y|^s}\right)\frac{1}{|x_0-y|^{n+s}}\,dxdy\\
		&\quad-Cl\int_{\mathbb{R}^n\setminus B_r}\int_{B_\frac{3r}{4}}g\left(\frac{l}{r^s}\right)\frac{1}{|x_0-y|^{n+s}}\,dxdy\\
		&=Cl|B_r|\mathrm{Tail}(u_+;x_0,r)-Clr^{-s}g(l/r^s)|B_r|,
		\end{align*}
		where we have used \eqref{0-3} and the facts that, for $x\in B_\frac{r}{2}$ and $y\in\mathbb{R}^n\setminus B_r$
		$$
		|x-y|\leq \left(1+\frac{|x-x_0|}{|y-x_0|}\right)|y-x_0|\leq 2|y-x_0|,
		$$
		and for $x\in B_\frac{3r}{4}$ and $y\in\mathbb{R}^n\setminus B_r$
		$$
		|y-x_0|\leq \left(1+\frac{|x-x_0|}{|y-x|}\right)|y-x|\leq 4|y-x|.
		$$
		
		On the other hand, with the help of \eqref{0-3} and \eqref{1-5}, we get
		\begin{align*}
		I_{22}
		&=2l\int_{B_R\setminus B_r}\int_{B_r}g\left(\frac{|u(y)-u(x)|}{|x-y|^s}\right)\chi_{\{u(y)<l\}}\varphi^q(x)\frac{K(x,y)}{|x-y|^s}\,dxdy\\
		&\quad+2l\int_{\mathbb{R}^n\setminus B_R}\int_{B_r}g\left(\frac{|u(y)-u(x)|}{|x-y|^s}\right)\chi_{\{u(y)<l\}}\varphi^q(x)\frac{K(x,y)}{|x-y|^s}\,dxdy\\
		&\leq 2l\int_{B_R\setminus B_r}\int_{B_r}g\left(\frac{l}{|x-y|^s}\right)\varphi^q(x)\frac{K(x,y)}{|x-y|^s}\,dxdy\\
		&\quad+2l\int_{\mathbb{R}^n\setminus B_R}\int_{B_r}g\left(\frac{l+u_-(y)}{|x-y|^s}\right)\varphi^q(x)\frac{K(x,y)}{|x-y|^s}\,dxdy\\
		&\leq Cl\int_{\mathbb{R}^n\setminus B_r}\int_{B_\frac{3r}{4}}g\left(\frac{l}{|x-y|^s}\right)\,\frac{dxdy}{|x-y|^{n+s}}+Cl\int_{\mathbb{R}^n\setminus B_R}\int_{B_r}g\left(\frac{u_-(y)}{|x-y|^s}\right)\,\frac{dxdy}{|x-y|^{n+s}}\\
		&\leq Clr^{-s}g(l/r^s)|B_r|+Cl|B_r|\mathrm{Tail}(u_-;x_0,R).
		\end{align*}
		As a result,
		\begin{equation}
		\label{3-2}
		I_2\geq Cl|B_r|\mathrm{Tail}(u_+;x_0,r)-Clr^{-s}g(l/r^s)|B_r|-Cl|B_r|\mathrm{Tail}(u_-;x_0,R).
		\end{equation}
		
		Next it remains to deal with the integral $I_1$. Set $v:=u-2l$. Suppose, without loss of generality, that $\varphi(x)\geq\varphi(y)$. Then $\varphi^q(x)-\varphi^q(y)\leq q\varphi^{q-1}(x)(\varphi(x)-\varphi(y))$. For $(x,y)\in B_r\times B_r$, we in turn employ the inequalities \eqref{0-3}, \eqref{1-2}-\eqref{1-4} to arrive at
		\begin{align*}
		&\quad g\left(\frac{|u(x)-u(y)|}{|x-y|^s}\right)\frac{u(x)-u(y)}{|u(x)-u(y)|}(v(x)\varphi^q(x)
		-v(y)\varphi^q(y))\frac{1}{|x-y|^s}\\
		&=g\left(\frac{|v(x)-v(y)|}{|x-y|^s}\right)\frac{|v(x)-v(y)|}{|x-y|^s}\varphi^q(x)\\
		&\quad+
		g\left(\frac{|v(x)-v(y)|}{|x-y|^s}\right)\frac{u(x)-u(y)}{|u(x)-u(y)|}\frac{\varphi^q(x)
			-\varphi^q(y)}{|x-y|^s}v(y)\\
		&\geq pG\left(\frac{|v(x)-v(y)|}{|x-y|^s}\right)\varphi^q(x)-qg\left(\frac{|v(x)-v(y)|}{|x-y|^s}\right)\varphi^{q-1}(x)\frac{|\varphi(x)-\varphi(y)|}{|x-y|^s}|v(y)|\\
		&\geq pG\left(\frac{|v(x)-v(y)|}{|x-y|^s}\right)\varphi^q(x)-\epsilon qG^*\left(g\left(\frac{|v(x)-v(y)|}{|x-y|^s}\right)\varphi^{q-1}(x)\right)\\
		&\quad-\epsilon^{q-1}qG\left(\frac{|\varphi(x)-\varphi(y)|}{|x-y|^s}|v(y)|\right)\\
		&\geq pG\left(\frac{|v(x)-v(y)|}{|x-y|^s}\right)\varphi^q(x)-\epsilon qG\left(\frac{|v(x)-v(y)|}{|x-y|^s}\right)\varphi^q(x)\\
		&\quad-\epsilon^{q-1}qG\left(\frac{|\varphi(x)-\varphi(y)|}{|x-y|^s}|v(y)|\right)\\
		&=\frac{p}{2}G\left(\frac{|v(x)-v(y)|}{|x-y|^s}\right)\varphi^q(x)-CG\left(\frac{|\varphi(x)-\varphi(y)|}{|x-y|^s}|v(y)|\right)\\
		&\geq -CG\left(l\frac{|\varphi(x)-\varphi(y)|}{|x-y|^s}\right).
		\end{align*}
		Here we need note $\varphi^{q-1}(x)\leq1$ and take $\epsilon=\frac{p}{2q}$. From this, we find that
		\begin{align}
		\label{3-3}
		I_1&\geq-C\int_{B_r}\int_{B_r}G\left(l\frac{|\varphi(x)-\varphi(y)|}{|x-y|^s}\right)\,\frac{dxdy}{|x-y|^n} \nonumber\\
		&\geq-C\int_{B_r}\int_{B_r}G\left(\frac{l}{r^s}\left(\frac{|x-y|}{r}\right)^{1-s}\right)\,\frac{dxdy}{|x-y|^n} \nonumber\\
		&\geq-C\int_{B_r}\int_{B_r}r^{(s-1)p}G(l/r^s)\,\frac{dxdy}{|x-y|^{n-(1-s)p}} \nonumber\\
		&\geq -CG(l/r^s)|B_r|.
		\end{align}
		Consequently, it holds, by combing \eqref{3-2}, \eqref{3-3} with \eqref{3-1}, that
		\begin{align*}
		\mathrm{Tail}(u_+;x_0,r)&\leq C\mathrm{Tail}(u_-;x_0,R)+C\frac{1}{l}G(l/r^s)\\
		&\leq Cr^{-s}g(l/r^s)+C\mathrm{Tail}(u_-;x_0,R).
		\end{align*}
		
		Finally, observe that for $a,b\geq0$ and $c\geq1$,
		\begin{align*}
		&a+b=g(g^{-1}(a))+g(g^{-1}(b))\leq2g(g^{-1}(a)+g^{-1}(b))\\
		\Rightarrow \ & g^{-1}\left(\frac{a+b}{2}\right)\leq g^{-1}(a)+g^{-1}(b)
		\end{align*}
		and
		\begin{equation}
		\label{3-3-1}
		g^{-1}(c^{-1}a)\geq (qc/p)^{-\frac{1}{p-1}}g^{-1}(a).
		\end{equation}
		Otherwise, by \eqref{0-3}, \eqref{1-1} and the strictly increasing property of $g$,
		\begin{align*}
		c^{-1}a&<g\left((qc/p)^{-\frac{1}{p-1}}g^{-1}(a)\right)\\
		&\leq q\frac{G\left((qc/p)^{-\frac{1}{p-1}}g^{-1}(a)\right)}{(qc/p)^{-\frac{1}{p-1}}g^{-1}(a)}\\
		&\leq q(qc/p)^{-1}\frac{G(g^{-1}(a))}{g^{-1}(a)}\leq c^{-1}a,
		\end{align*}
		which is a contradiction.
		Then we have
		\begin{align*}
		g^{-1}((2C)^{-1}r^s\mathrm{Tail}(u_+;x_0,r))&\leq g^{-1}\left(\frac{g(l/r^s)+r^s\mathrm{Tail}(u_-;x_0,R)}{2}\right)\\
		&\leq\frac{l}{r^s}+g^{-1}(r^s\mathrm{Tail}(u_-;x_0,R))
		\end{align*}
		and
		$$
		g^{-1}((2C)^{-1}r^s\mathrm{Tail}(u_+;x_0,r))\geq \frac{1}{C}g^{-1}(r^s\mathrm{Tail}(u_+;x_0,r)),
		$$
		which means the desired result.
	\end{proof}
	
	In order to infer Harnack inequality for Eq. \eqref{main}, we need the following local boundedness result on weak subsolutions that is a slightly modified version of \cite[Theorem 4.4]{BKO21}.
	
	\begin{lemma}
		\label{lem3-2}
		Let $B_r(x_0)\subset\subset\Omega$. Assume that $u\in \mathbb{W}^{s,G}(\Omega)\cap L^g_s(\mathbb{R}^n)$ is a weak subsolution to Eq. \eqref{main}. Then there holds that
		\begin{equation}
		\label{3-4}
		\sup_{B_\frac{r}{2}}u\leq Cr^sG^{-1}\left(\delta^\frac{\theta}{1-\theta}\mint_{B_r}G\left(\frac{u_+}{r^s}\right)\,dx\right)
		+\left(\frac{r}{2}\right)^sg^{-1}\left(\delta \left(\frac{r}{2}\right)^s\mathrm{Tail}\left(u_+;x_0,\frac{r}{2}\right)\right),
		\end{equation}
		where $C$ depends on $n,p,q,s,\Lambda$.
	\end{lemma}
	
	\begin{proof}
		The process is the same as that of \cite[Theorem 4.4]{BKO21}. Let us point out that the notations below adopt identically those in \cite[Theorem 4.4]{BKO21}. We just need to notice that, after the inequality (4.14) in \cite{BKO21}, the parameter $k$ is first chosen so large that
		$$
		k\geq \left(\frac{r}{2}\right)^sg^{-1}\left(\delta \left(\frac{r}{2}\right)^s\mathrm{Tail}\left(u_+;x_0,\frac{r}{2}\right)\right)
		$$
		with $\delta\in(0,1]$, instead of the value of $k$ there. Then the inequality (4.14) in \cite{BKO21} becomes
		$$
		a_{j+1}\leq C2^{j(n+sq+2q)\theta}\left(1+\frac{1}{\delta}\right)^\theta a^\theta_j\leq 2^\theta C2^{j(n+sq+2q)\theta}\delta^{-\theta} a^\theta_j.
		$$
		Let $C_0=2^\theta C$ and $B=2^{(n+sq+2q)\theta}$. Then
		$$
		a_{j+1}\leq(\delta^{-\theta}C_0)B^ja^\theta_j.
		$$
		By the iteration lemma (see, e.g. \cite[Lemma 7.1]{Giu}), we need require
		$$
		a_0\leq (\delta^{-\theta}C_0)^{-\frac{1}{\theta-1}}B^{-\frac{1}{(\theta-1)^2}},
		$$
		namely,
		$$
		\frac{\mint_{B_r}G\left(\frac{u_+}{r^s}\right)\,dx}{G(k/r^s)}\leq (\delta^{-\theta}C_0)^{-\frac{1}{\theta-1}}B^{-\frac{1}{(\theta-1)^2}},
		$$
		so that $a_j\rightarrow0$ as $j\rightarrow\infty$. Now we pick
		$$
		k=r^sG^{-1}\left(\delta^\frac{\theta}{1-\theta}C_0^{\frac{1}{\theta-1}}B^{\frac{1}{(\theta-1)^2}}\mint_{B_r}G\left(\frac{u_+}{r^s}\right)\,dx\right)
		+\left(\frac{r}{2}\right)^sg^{-1}\left(\delta \left(\frac{r}{2}\right)^s\mathrm{Tail}\left(u_+;x_0,\frac{r}{2}\right)\right).
		$$
		Terminally, the limit $\lim_{j\rightarrow\infty}a_j=0$ leads to \eqref{3-4}.
	\end{proof}
	
	Finally, we implement the proof of the nonlocal Harnack inequality stated in Theorem \ref{thm0}. From this procedure, one can apparently understand the reason why we impose the conditions \eqref{0-4} and \eqref{0-5}.
	
	\medskip
	
	\noindent\textbf{Proof of Theorem \ref{thm0}.} For simplicity, let $\lambda=\frac{\theta}{1-\theta}$. Putting together the local boundedness estimate (Lemma \ref{lem3-2}) and the tail estimate (Lemma \ref{lem3-1}), we derive that, for $B_\rho\subset\subset\Omega$,
	\begin{align*}
	\sup_{B_{\frac{\rho}{2}}}u&\leq C\rho^sG^{-1}\left(\delta^\lambda\mint_{B_\rho}G\left(\frac{u_+}{\rho^s}\right)\,dx\right)
	+C\delta^\frac{1}{q-1}\left(\frac{\rho}{2}\right)^sg^{-1}\left( \left(\frac{\rho}{2}\right)^s\mathrm{Tail}\left(u_+;x_0,\frac{\rho}{2}\right)\right)\\
	&\leq C\delta^\frac{1}{q-1}\sup_{B_\rho}u+C\rho^sG^{-1}\left(\delta^\lambda\mint_{B_\rho}G\left(\frac{u_+}{\rho^s}\right)\,dx\right)
	+C\delta^\frac{1}{q-1}\rho^sg^{-1}(\rho^s\mathrm{Tail}(u_-;x_0,R)).
	\end{align*}
	Here we have utilized
	\begin{equation*}
	\begin{cases}
	g^{-1}(at)\leq (q/p)^\frac{1}{p-1}a^\frac{1}{q-1}g^{-1}(t)  &\text{\textmd{for }} 0<a<1, t\geq0,\\[2mm]
	g^{-1}(at)\leq (q/p)^\frac{1}{p-1}a^\frac{1}{p-1}g^{-1}(t)  &\text{\textmd{for }} a\geq1, t\geq0,
	\end{cases}
	\end{equation*}
	which can be justified in a similar way to \eqref{3-3-1}. 
	
	We next would like to apply the iteration \cite[Lemma 1]{GG82} (See also \cite[Lemma 2.7]{DKP14}). Denote $\rho=(\gamma-\gamma')r$ with $\frac{1}{2}\leq\gamma'<\gamma\leq1$. By a covering argument, we obtain
	\begin{align}
	\label{3-5}
	\sup_{B_{\gamma'r}}u&\leq C\frac{r^{s(1-\frac{q}{p})}}{(\gamma-\gamma')^{\frac{n}{p}+s(\frac{q}{p}-1)}}G^{-1}\left(\delta^\lambda\mint_{B_{\gamma r}}G(u)\,dx\right) \nonumber\\
	&\quad+C\delta^\frac{1}{q-1}r^sg^{-1}(r^s\mathrm{Tail}(u_-;x_0,R))+C\delta^\frac{1}{q-1}\sup_{B_{\gamma r}}u,
	\end{align}
	where we note the positivity of $u$ in $B_R(x_0)$. Now making use of \eqref{0-4} and \eqref{0-5}, we evaluate, for any $\varepsilon\in(0,1)$,
	\begin{align}
	\label{3-6}
	&\quad G^{-1}\left(\delta^\lambda\mint_{B_{\gamma r}}G(u)\,dx\right) \nonumber\\
	&\leq G^{-1}\left(c_0\delta^\lambda\mint_{B_{\gamma r}}G(u^{1-\varepsilon})G(u^\varepsilon)\,dx\right) \nonumber\\
	&\leq CG^{-1}\left(\delta^\lambda\mint_{B_{\gamma r}}G\left(\left(\sup_{B_{\gamma r}}u\right)^{1-\varepsilon}\right)G(u^\varepsilon)\,dx\right) \nonumber\\
	&\leq C(\delta)\left(\sup_{B_{\gamma r}}u\right)^{1-\varepsilon}\max\left\{\left(\mint_{B_{\gamma r}}G(u^\varepsilon)\,dx\right)^\frac{1}{p},\left(\mint_{B_{\gamma r}}G(u^\varepsilon)\,dx\right)^\frac{1}{q}\right\}.
	\end{align}
	Via selecting $\delta=\left(\frac{1}{4C}\right)^{q-1}$, merging the displays \eqref{3-5}, \eqref{3-6} and an application of Young's inequality, we have
	\begin{align*}
	\sup_{B_{\gamma'r}}u&\leq \left(\sup_{B_{\gamma r}}u\right)^{1-\varepsilon}\frac{Cr^{s(1-\frac{q}{p})}}{(\gamma-\gamma')^{\frac{n}{p}+s(\frac{q}{p}-1)}}\max\left\{\left(\mint_{B_{\gamma r}}G(u^\varepsilon)\,dx\right)^\frac{1}{p},\left(\mint_{B_{\gamma r}}G(u^\varepsilon)\,dx\right)^\frac{1}{q}\right\}\\
	&\quad+\frac{1}{4}\sup_{B_{\gamma r}}u+Cr^sg^{-1}(r^s\mathrm{Tail}(u_-;x_0,R))\\
	&\leq \frac{1}{2}\sup_{B_{\gamma r}}u+\frac{Cr^{s(1-\frac{q}{p})\frac{1}{\varepsilon}}}{(\gamma-\gamma')^{[\frac{n}{p}+s(\frac{q}{p}-1)]\frac{1}{\varepsilon}}}
	\left[\max\left\{\left(\mint_{B_r}G(u^\varepsilon)\,dx\right)^\frac{1}{p},\left(\mint_{B_{r}}G(u^\varepsilon)\,dx\right)^\frac{1}{q}\right\}\right]^\frac{1}{\varepsilon}\\
	&\quad+Cr^sg^{-1}(r^s\mathrm{Tail}(u_-;x_0,R)).
	\end{align*}
	We can apply \cite[Lemma 1]{GG82} to infer that
	\begin{align*}
	\sup_{B_r}u&\leq Cr^{s(1-\frac{q}{p})\frac{1}{\varepsilon}}
	\left[\max\left\{\left(\mint_{B_r}G(u^\varepsilon)\,dx\right)^\frac{1}{p},\left(\mint_{B_{r}}G(u^\varepsilon)\,dx\right)^\frac{1}{q}\right\}\right]^\frac{1}{\varepsilon}\\
	&\quad+Cr^sg^{-1}(r^s\mathrm{Tail}(u_-;x_0,R)).
	\end{align*}
	In order to make use of Lemma \ref{lem2-3}, we have to impose the conditions that $G(t)\leq c_1\max\{t^p,t^q\}$. 
	
	We proceed by considering the integral term in the above display,
	\begin{align*}
	\mint_{B_r}G(u^\varepsilon)\,dx&=\frac{1}{|B_r|}\left(\int_{B_r\cap\{u<1\}}G(u^\varepsilon)\,dx+\int_{B_r\cap\{u\geq1\}}G(u^\varepsilon)\,dx\right)\\
	&\leq \frac{c_1}{|B_r|}\left(\int_{B_r}u^{p\varepsilon}\,dx+\int_{B_r}u^{q\varepsilon}\,dx\right)\\
	&\leq c_1\left(\mint_{B_r}u^{q\varepsilon}\,dx\right)^\frac{p}{q}+c_1\mint_{B_r}u^{q\varepsilon}\,dx.
	\end{align*}
	Combining the last two displays and choosing $\varepsilon=\frac{\epsilon}{q}$ with $\epsilon$ given by Lemma \ref{lem2-3}, we finally arrive at
	\begin{align*}
	\sup_{B_r}u
	&\leq Cr^{s(1-\frac{q}{p})\frac{q}{\epsilon}}\max_{i\in\{\frac{1}{\epsilon},\frac{q}{p\epsilon}\}}
	\Bigg\{\left(\inf_{B_r}u+r^sg^{-1}(r^s\mathrm{Tail}(u_-;x_0,R))\right)^\frac{ip\epsilon}{q}\\
	&\qquad\qquad\qquad\qquad\qquad+\left(\inf_{B_r}u+r^sg^{-1}(r^s\mathrm{Tail}(u_-;x_0,R))\right)^{i\epsilon}\Bigg\}\\
	&\quad+Cr^sg^{-1}(r^s\mathrm{Tail}(u_-;x_0,R))\\
	&\leq Cr^{s(1-\frac{q}{p})\frac{q}{\epsilon}}\max_{\iota\in\{1,\frac{q}{p},\frac{p}{q}\}}\left\{\left(\inf_{B_r}u+r^sg^{-1}(r^s\mathrm{Tail}(u_-;x_0,R))\right)^\iota\right\}\\
	&\quad+Cr^sg^{-1}(r^s\mathrm{Tail}(u_-;x_0,R)).
	\end{align*}
	The proof is complete now.

	\section*{Acknowledgments}
	
	This work was supported by the National Natural Science Foundation of China (No. 12071098).

\end{document}